\documentclass[a4paper,12pt]{article}
\usepackage[top=2.5cm,bottom=2.5cm,left=2.5cm,right=2.5cm]{geometry}
\usepackage{cite, amsmath, amssymb}
\usepackage{amssymb}
\usepackage{graphicx}
\newenvironment{proof}{{\noindent \it Proof.}}{\hfill $\blacksquare$\par}
\usepackage{latexsym}
\usepackage{graphicx,booktabs,multirow}
\usepackage{tikz}
\usetikzlibrary{decorations.pathreplacing}
\usetikzlibrary{intersections}
\usepackage{enumerate}
\usepackage{graphicx,booktabs,multirow}
\usepackage{appendix}
\newtheorem{theorem}{Theorem}[section]
\newtheorem{proposition}[theorem]{\rm\bfseries Proposition}
\newtheorem{lemma}[theorem]{Lemma}

\usepackage[section]{placeins}
\allowdisplaybreaks

\begin{document}
\vspace*{10mm}

\noindent
{\Large \bf The Sombor index and coindex of two-trees}

\vspace*{7mm}

\noindent
{\large \bf Zenan Du$^1$, Lihua You$^{1,*}$, Hechao Liu$^1$, Yufei Huang$^2$}
\noindent

\vspace{7mm}

\noindent
$^1$ School of Mathematical Sciences, South China Normal University,  Guangzhou, 510631, P. R. China,
e-mail: {\tt 2021010121@m.scnu.edu.cn},\quad {\tt ylhua@scnu.edu.cn}, \\ \quad{\tt hechaoliu@m.scnu.edu.cn},\\[2mm]
$^2$ Department of Mathematics Teaching, Guangzhou Civil Aviation College, Guangzhou, 510403, P. R. China,
e-mail: {\tt fayger@qq.com} \\[2mm]
$^*$ Corresponding author
\noindent

\vspace{7mm}

\noindent
{\bf Abstract} \
\noindent
The Sombor index of a graph $G$, introduced by Ivan Gutman, is defined as the sum of the weights $\sqrt{d_G(u)^2+d_G(v)^2}$ of all edges $uv$ of $G$, where $d_G(u)$ denotes the degree of vertex $u$ in $G$. The Sombor coindex is recently defined as $\overline{SO}(G)=\sum \limits_{uv\notin E(G)}\sqrt{d_G(u)^2+d_G(v)^2}$. In this paper, the maximum and second maximum Sombor index, the minimum and second minimum Sombor coindex in two-trees are determined.


\noindent
{\bf MR(2020) Subject Classification} \ 05C09; 05C92
\baselineskip=0.30in

\section{Introduction}
\hspace{1.5em}Throughout this paper, let $G$ be a simple graph with vertex set $V(G)$ and edge set $E(G)$. Let $|V(G)|, |E(G)|$ denote the order and the number of edges of $G$, respectively. If two vertices $u, v$ are adjacent, we write $u\sim v$. Let $N_G(v)$ be the set of all vertices adjacent to $v$ and $d_G(v)=|N_G(v)|$ be the degree of a vertex $v\in V(G)$. If there is no confusion from the context, we abbreviate $d_G(v)$ by $d(v)$. The complement $\overline{G}$ of a graph $G$ is the graph with vertex set $V(G)$, in which two vertices are adjacent if and only if they are not adjacent in $G$. $K_n$, $K_{a,b}$ are complete and complete bipartite graph of order $n$, $a+b$, respectively.

In the chemical and pharmaceutical sciences, topological indices are graph invariants that play a significant role\cite{Trinajstic}. Molecular models can be used to study chemical graphs in which vertices represent atoms and edges between vertices represent chemical bonds. In recent years, many topological indices have been introduced and applied in various fields of science including structural chemistry, theoretical chemistry, environmental chemistry, etc \cite{Gao, Dearden}. There are several types of topological indices, one of the most important being vertex-degree-based topological indices. For instance, Ivan Gutman introduces the Sombor indices, defined as \cite{Gutman-Geometric}
\begin{equation*}
  SO(G)=\sum_{uv\in E(G)}\sqrt{d_G(u)^2+d_G(v)^2},
\end{equation*}
and the Sombor coindex is recently defined as \cite{Chinglensana}
\begin{equation*}
  \overline{SO}(G)=\sum_{uv\notin E(G)}\sqrt{d_G(u)^2+d_G(v)^2}.
\end{equation*}

Determining the extreme value of topological indices of different graph classes has always been the focus and hot research issue of chemical graph theory, and its results can be used for analysis the structure of compounds with physical or chemical properties provides mathematical methods and tools.
It is verified that the Sombor index has good prediction and identification ability for the simulation of alkane vaporization entropy and vaporization enthalpy \cite{Redzepovic}. At present, the research on topological coindices mainly includes graph operations, coindices of graphs and (in)equalities, the correlation results can be found in \cite{Ashrafi, Ghalavand, Liu-JB, Belay, Gutman-Zagreb}.

The two-tree $T_t$ is defined in \cite{Cochrane} as follows: (1) $T_0\cong K_2$ where $K_2$ is a two-tree with 2 vertices; (2) $T_t (t\geq 1)$ is a two-tree obtained from $T_{t-1}$ by adding a new vertex adjacent to the two end vertices of one edge. Two-trees has a very important structure in complex networks, such as generalized Farey graph and fractal scale-free network \cite{Zhang}. At present, the extreme values of many indices of two-trees have been determined (see \cite{Liu-kq, Sun, Sun-xl, yu, Yu}), while the extreme values of Somber (co)index of two trees are unknown. We continue to determine the extreme value of Sombor (co)index of two-trees.

\section{Extreme Sombor index and Sombor coindex of two-trees}\label{sec-two-trees}
\hspace{1.5em}In this section, the maximum and second maximum Sombor index, the minimum and second minimum Sombor coindex in two-trees are determined by mathematical induction and analytical structure method.

It is clear that $T_t$ $(t\geq 1)$ has at least two vertices of degree 2 and $|V(T_t)|=t+2$. Let $X_n$ denote the graph obtained from $K_{2,n-2}$ by joining an edge between the two vertices of degree $n-2$ and $L_n$ denote the graph obtained from $X_{n-1}$ by adding a new vertex and edge it with a 2-degree point and a $(n-2)$-degree point, then both $X_n$, $L_n$ (see Figure \ref{Xn}, Figure \ref{Ln}) are two-trees.
\begin{figure}[htbp]
\centering
\begin{minipage}[t]{0.48\textwidth}
\centering
\includegraphics[width=3.75cm]{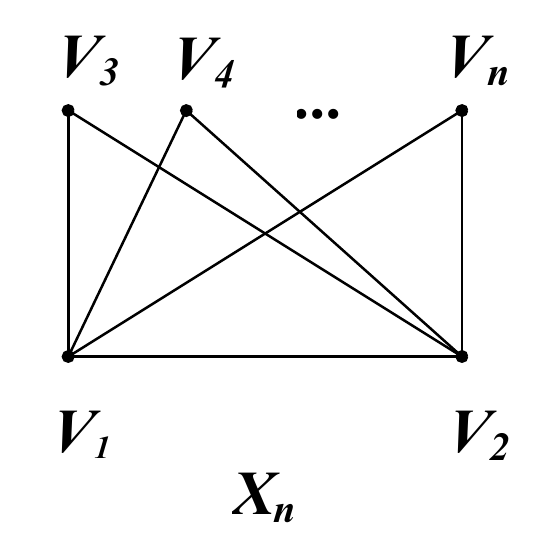}
\caption{The graph $X_n$}\label{Xn}
\end{minipage}
\begin{minipage}[t]{0.48\textwidth}
\centering
\includegraphics[width=5cm]{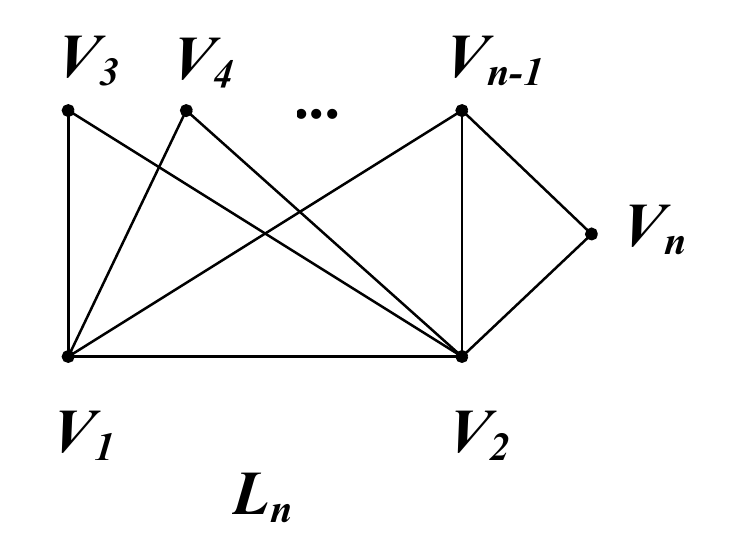}
\caption{The graph $L_n$}\label{Ln}
\end{minipage}
\end{figure}

We first prove some lemmas in preparation for the following main results.

From the definition of the Sombor index and Sombor coindex, we can directly get the following results.
\begin{proposition}
The Sombor index and Sombor coindex of two-tree $X_n$ and $L_n$ are
\begin{equation*}
  \begin{split}
    SO(X_n)=&\sqrt{2}(n-1)+2(n-2)\sqrt{(n-1)^2+4}, \\ \overline{SO}(X_n)=&\sqrt{2}(n-2)(n-3),\\
    SO(L_n)=&(n-4)\sqrt{(n-2)^2+4}+(n-3)\sqrt{(n-1)^2+4}+\sqrt{(n-1)^2+(n-2)^2}\\
    &+\sqrt{13}+\sqrt{(n-1)^2+9}+\sqrt{(n-2)^2+9},\\
    \overline{SO}(L_n)=&\sqrt{2}(n-3)(n-4)+\sqrt{(n-2)^2+4}+\sqrt{13}(n-4).
  \end{split}
\end{equation*}
\end{proposition}

\begin{lemma}\label{lemma1}
  Let $f(x,y)=\sqrt{x^2+y^2}-\sqrt{(x-1)^2+(y-1)^2}$ where $x,y\geq2$. Then $f(x,y)\leq \sqrt{2}$ with equality holds if and only if $x=y$. Moreover, $\frac{\partial f}{\partial x}$ is monotonic increasing with $x$ if $x<y$.
\end{lemma}
\begin{proof}
  We prove this conclusion in two-dimensional plane rectangular coordinate system. It is clear that $\sqrt{x^2+y^2}$, $\sqrt{(x-1)^2+(y-1)^2}$ represents the distance between the coordinate $(x,y)$ and $(0,0)$, $(1,1)$ respectively. Then points $(x,y)$, $(0,0)$, $(1,1)$ will form a triangle and let $e_1, e_2, e_3$ represent the edge after connecting points $(x,y)$ and $(0,0)$, $(x,y)$ and $(1,1)$, $(0,0)$ and $(1,1)$, respectively. At this time, the problem is transformed into solving the maximum value of $|e_1|-|e_2|$ where $|e_i|$ represents the length of $e_i$ $(i=1,2,3)$. From the properties of triangles, we know that the difference value between the lengths of any two edges is less than the third edge. Thus $|e_1|-|e_2|\leq |e_3|=\sqrt{2}$ with equality holds if and only if $(x,y)$, $(0,0)$, $(1,1)$ are collinear.

  When $y>x\geq2$, we have $x^2(y-1)^2-(x-1)^2y^2=(2xy-(x+y))(y-x)>0$. Then
  \begin{equation*}
  \begin{split}
    \frac{\partial f}{\partial x}=&\frac{x}{\sqrt{x^2+y^2}}-\frac{x-1}{\sqrt{(x-1)^2+(y-1)^2}}\\
    =&\frac{\sqrt{x^2(x-1)^2+x^2(y-1)^2}-\sqrt{x^2(x-1)^2+(x-1)^2y^2}}{\sqrt{x^2+y^2}\sqrt{(x-1)^2+(y-1)^2}}>0,
  \end{split}
  \end{equation*}
  and we prove the remainder of the Lemma.
\end{proof}

\begin{lemma}\label{lemma2}
  Let $f(x,y)=\sqrt{x^2+y}-\sqrt{(x-1)^2+y}$, $g(x,y)=f(x,y)-f(x-1,y)$, where $x,y$ are positive integers. Then
  \begin{itemize}
  \item[\rm(i)] $f(x,y)$ is monotonic increasing with $x$;
  \item[\rm(ii)] $f(x,y)$ is monotonic decreasing with $y$;
  \item[\rm(iii)] $g(x,y)$ is monotonic decreasing with $x$ and $g(x,y)>0$.
  \end{itemize}
\end{lemma}

\begin{proof}
  We consider the derivative of $f(x,y)$,
  \begin{equation*}
  \begin{split}
    \frac{\partial f}{\partial x}=&\frac{x}{\sqrt{x^2+y}}-\frac{x-1}{\sqrt{(x-1)^2+y}}\\
         =&\frac{\sqrt{x^2(x-1)^2+yx^2}-\sqrt{x^2(x-1)^2
         +y(x-1)^2}}{\sqrt{x^2+y}\sqrt{(x-1)^2+y}}>0,
  \end{split}
  \end{equation*}
  \begin{equation*}
    \frac{\partial f}{\partial y}=\frac{1}{2}\left(\frac{1}{\sqrt{x^2+y}}-\frac{1}{\sqrt{(x-1)^2+y}}\right)
      =\frac{\sqrt{(x-1)^2+y}-\sqrt{x^2+y}}{\sqrt{x^2+y}\sqrt{(x-1)^2+y}}
      <0.
  \end{equation*}

  Hence $f(x,y)$ is monotonically increases with $x$ and monotonically decreases with $y$, and $g(x,y)>0$ is naturally.

  Besides, by
  \begin{equation*}
    \frac{\partial^2 f}{\partial x^2}=\partial_x(\frac{x}{\sqrt{x^2+y}})-\partial_x(\frac{x-1}{\sqrt{(x-1)^2+y}})
         =y(\frac{1}{(x^2+y)^\frac{3}{2}}-\frac{1}{((x-1)^2+y)^\frac{3}{2}})
         <0
  \end{equation*}
  we know that $\frac{\partial f}{\partial x}$ is monotonic decreasing with $x$ and thus the third claim holds.
\end{proof}

Now we determine the maximum Sombor index of two-trees.
\begin{theorem}
Let $G_n$ be a two-tree of order $n\geq 2$. Then
\begin{equation*}
SO(G_n)\leq \sqrt{2}(n-1)+2(n-2)\sqrt{(n-1)^2+4},
\end{equation*}
with equality holds if and only if $G_n\cong X_n$.
\end{theorem}
\begin{proof}
We prove this result by induction on $n$.

If $n=2,3,4$, $G_n\cong X_2, X_3, X_4$ respectively, so the equality holds.

Assume that the result holds for $n-1$. Choose a vertex $w$ of degree 2 from the graph $G_n$, then $G_n-w$ is a two-tree of order $n-1$. By the induction hypothesis, $SO(G_n-w)\leq SO(X_{n-1})$ with equality holds if and only if $G_n-w \cong X_{n-1}$. In the following we prove that $SO(G_n)\leq SO(X_n)$.

Let $u$ and $v$ be two vertices adjacent to the vertex $w$ in $G_n$. Let $d(u)=a,d(v)=b$ and $N_{G_n}(u)\setminus\{v,w\}=\{u_1,u_2,\cdots,u_{a-2}\}$, $N_{G_n}(v)\setminus\{u,w\}=\{v_1,v_2,\cdots,v_{b-2}\}$.
From the construction of $G_n$ we know that $3\leq a,b\leq n-1$, then
\begin{align*}
   SO(G_n)=&SO(G_n-w)+\sqrt{a^2+4}+\sqrt{b^2+4}+\sqrt{a^2+b^2}-\sqrt{(a-1)^2+(b-1)^2}\\
          &+\sum_{i=1}^{a-2}\Big(\sqrt{a^2+d(u_i)^2}-\sqrt{(a-1)^2+d(u_i)^2}\Big)\\
          &+\sum_{j=1}^{b-2}\Big(\sqrt{b^2+d(v_j)^2}-\sqrt{(b-1)^2+d(v_j)^2}\Big)\\
          \leq &SO(X_{n-1})+2\sqrt{(n-1)^2+4}+\sqrt{2}\\
          &+\sum_{i=1}^{n-3}\Big(\sqrt{(n-1)^2+d(u_i)^2}-\sqrt{(n-2)^2+d(u_i)^2}\Big)\\
          &+\sum_{j=1}^{n-3}\Big(\sqrt{(n-1)^2+d(v_j)^2}-\sqrt{(n-2)^2+d(v_j)^2}\Big)\\
          \leq&SO(X_{n-1})+2\sqrt{(n-1)^2+4}+\sqrt{2}\\
          &+2(n-3)\sqrt{(n-1)^2+4}-2(n-3)\sqrt{(n-2)^2+4}\\
          =&2(n-3)\sqrt{(n-2)^2+4}+\sqrt{2}(n-2)+\sqrt{2}\\
          &+2(n-2)\sqrt{(n-1)^2+4}-2(n-3)\sqrt{(n-2)^2+4}\\
          =&2(n-2)\sqrt{(n-1)^2+4}+\sqrt{2}(n-1)\\
          =&SO(X_n).
\end{align*}

The first and second inequalities follows from Lemma \ref{lemma1} and Lemma \ref{lemma2}. Hence the equality holds if and only if $G_n-w\cong X_{n-1}$, $a=b=n-1$ and $d(u_i)=d(v_i)=2(i=1,2,\cdots,n-3)$, which implies $G_n\cong X_n$ and we complete the proof.
\end{proof}

Next we determine the second maximum Sombor index of two-trees.
\begin{theorem}
  Let $G_n$ be a two-tree of order $n\geq 5$ and $G_n\ncong X_n$. Then
\begin{equation*}
\begin{split}
SO(G_n)\leq& (n-4)\sqrt{(n-2)^2+4}+(n-3)\sqrt{(n-1)^2+4}+\sqrt{(n-1)^2+(n-2)^2}\\
&+\sqrt{13}+\sqrt{(n-1)^2+9}+\sqrt{(n-2)^2+9},
\end{split}
\end{equation*}
with equality holds if and only if $G_n\cong L_n$.
\end{theorem}

\begin{proof}
We prove this result by induction on $n$.

If $n=5$, $G_5$ can only be isomorphic to $L_5$ and $X_5$, so the equality holds.

Assume that the result holds for $n-1$. We choose one vertex $w$ of degree 2 from $G_n$ such that $G_n-w \ncong X_{n-1}$, then $G_{n}-w$ is a two-tree of order $n-1$. By the induction hypothesis, $SO(G_n-w)\leq SO(L_{n-1})$ with equality holds if and only if $G_n-w \cong L_{n-1}$. In the following we prove that $SO(G_n)\leq SO(L_n)$.

Let $u$ and $v$ be two vertices adjacent to the vertex $w$ in $G_n$. Since $n\geq 5$, from the definition of two-tree we known that there must exist a vertex $p$ $(d(p)\geq3)$ which is adjacent to $u$ and $v$(otherwise $G_n-w \cong X_{n-1}$). Let $d(u)=a,d(v)=b,d(p)=c$ and $N_{G_n}(u)\setminus \{v,w,p\}=\{u_1,u_2,\cdots,u_{a-3}\}$, $N_{G_n}(v)\setminus \{u,w,p\}=\{v_1,v_2,\cdots,v_{b-3}\}$. Then $3\leq a,b,c\leq n-1$. Without loss of generality we assume that $a\leq b$.

\textbf{Case 1}: $\max\{b,c\}=c$.

Then $c\leq n-2$ since $p\nsim w$, thus $a\leq b\leq n-2$, and
\begin{align*}
   SO(G_n)=&SO(G_n-w)+\sqrt{a^2+4}+\sqrt{b^2+4}+\sqrt{a^2+b^2}-\sqrt{(a-1)^2+(b-1)^2}\\
          &+\sqrt{a^2+c^2}-\sqrt{(a-1)^2+c^2}+\sqrt{b^2+c^2}-\sqrt{(b-1)^2+c^2}\\
          &+\sum_{i=1}^{a-3}\Big(\sqrt{a^2+d(u_i)^2}-\sqrt{(a-1)^2+d(u_i)^2}\Big)\\
          &+\sum_{j=1}^{b-3}\Big(\sqrt{b^2+d(v_j)^2}-\sqrt{(b-1)^2+d(v_j)^2}\Big)\\
          \leq &SO(L_{n-1})+2\sqrt{(n-2)^2+4}+\sqrt{2}+2\sqrt{(n-2)^2+9}-2\sqrt{(n-3)^2+9}\\
          &+2(n-5)\sqrt{(n-2)^2+4}-2(n-5)\sqrt{(n-3)^2+4}\\
          =&(n-5)\sqrt{(n-3)^2+4}+(n-4)\sqrt{(n-2)^2+4}+\sqrt{(n-2)^2+(n-3)^2}\\
          &+\sqrt{(n-2)^2+9}+\sqrt{(n-3)^2+9}+\sqrt{13}+2\sqrt{(n-2)^2+4}\\
          &+\sqrt{2}+2\sqrt{(n-2)^2+9}-2\sqrt{(n-3)^2+9}\\
          &+2(n-5)\sqrt{(n-2)^2+4}-2(n-5)\sqrt{(n-3)^2+4}\\
          =&3(n-4)\sqrt{(n-2)^2+4}-(n-5)\sqrt{(n-3)^2+4}+\sqrt{(n-2)^2+(n-3)^2}\\
          &+3\sqrt{(n-2)^2+9}-\sqrt{(n-3)^2+9}+\sqrt{2}+\sqrt{13}.\quad(1)
\end{align*}

Note that
\begin{equation*}
\begin{split}
SO(L_n)-(1)=&(n-3)\sqrt{(n-1)^2+4}+(n-5)\sqrt{(n-3)^2+4}-2(n-4)\sqrt{(n-2)^2+4}\\
&+\sqrt{(n-1)^2+(n-2)^2}-\sqrt{(n-2)^2+(n-3)^2}+\sqrt{(n-1)^2+9}\\
&+\sqrt{(n-3)^2+9}-2\sqrt{(n-2)^2+9}-\sqrt{2}.
\end{split}
\end{equation*}

Let $x=n-1, y=4$ $(n\geq 6)$ in Lemma \ref{lemma2} we have
\begin{equation*}
\begin{split}
&(n-3)\sqrt{(n-1)^2+4}+(n-5)\sqrt{(n-3)^2+4}-2(n-4)\sqrt{(n-2)^2+4}\\
=&2f(n-1,4)+(n-5)g(n-1,4)\geq 2f(5,4)+g(5,4)=1.8725.
\end{split}
\end{equation*}

Also from Lemma \ref{lemma2} it is easy to check that for $n\geq 6$,
\begin{equation*}
\sqrt{(n-1)^2+(n-2)^2}-\sqrt{(n-2)^2+(n-3)^2}\geq 1.40312,\\
\end{equation*}
and
\begin{equation*}
\sqrt{(n-1)^2+9}+\sqrt{(n-3)^2+9}-2\sqrt{(n-2)^2+9}=g(n-1,9)\geq0.
\end{equation*}
Hence $SO(G_n)\leq (1)<SO(L_n)$.

\textbf{Case 2}: $\max\{b,c\}=b$.

Then we have $b\leq n-1$ and $\max\{a,c\}\leq n-2$ (otherwise $G_n\cong X_n$). Thus
\begin{align*}
SO(G_n)=&SO(G_n-w)+\sqrt{a^2+4}+\sqrt{b^2+4}+\sqrt{a^2+b^2}-\sqrt{(a-1)^2+(b-1)^2}\\
          &+\sqrt{a^2+c^2}-\sqrt{(a-1)^2+c^2}+\sqrt{b^2+c^2}-\sqrt{(b-1)^2+c^2}\\
          &+\sum_{i=1}^{a-3}\Big(\sqrt{a^2+d(u_i)^2}-\sqrt{(a-1)^2+d(u_i)^2}\Big)\\
          &+\sum_{j=1}^{b-3}\Big(\sqrt{b^2+d(v_j)^2}-\sqrt{(b-1)^2+d(v_j)^2}\Big)\quad(2)\\
          \leq&SO(L_{n-1})+\sqrt{(n-2)^2+4}+\sqrt{(n-1)^2+4}\\
          &+\sqrt{(n-1)^2+(n-2)^2}-\sqrt{(n-2)^2+(n-3)^2}+\sqrt{(n-2)^2+9}\\
          &-\sqrt{(n-3)^2+9}+\sqrt{(n-1)^2+9}-\sqrt{(n-2)^2+9}\\
          &+(n-5)\sqrt{(n-2)^2+4}-(n-5)\sqrt{(n-3)^2+4}\\
          &+(n-4)\sqrt{(n-1)^2+4}-(n-4)\sqrt{(n-2)^2+4}\\
          =&(n-3)\sqrt{(n-1)^2+4}+(n-4)\sqrt{(n-2)^2+4}+\sqrt{(n-1)^2+(n-2)^2}\\
          &+\sqrt{13}+\sqrt{(n-1)^2+9}+\sqrt{(n-2)^2+9}\\
          =&SO(L_n).
\end{align*}

From Lemma \ref{lemma1} and Lemma \ref{lemma2} we can deduce that the sum at the right side of the first equality $(2)$, except the $SO(G_n-w)$ term is monotonically increasing with respect to $a$ and monotonically decreasing with respect to $c,d(u_i)$ and $d(v_i)$. Hence the maximum value is taken when $a=n-2, c=3, d(u_i)=d(v_i)=2$ for all $i$.

Next we use Lemma \ref{lemma2} repeatedly to prove the sum has maximum value when $b=n-1$. We consider the part of the sum containing b. We only need to compare the case $b=n-1$ and $b=n-2$, because when b is less than $n-2$, the sum will be smaller than $b=n-2$. For $b=n-1$, $b=n-2$, these sums are
\begin{equation*}
\begin{split}
  &\sqrt{(n-1)^2+4}+\big(\sqrt{(n-2)^2+(n-1)^2}-\sqrt{(n-3)^2+(n-2)^2}\big)+\big(\sqrt{(n-1)^2+9}\\
  &-\sqrt{(n-2)^2+9}\big)+(n-4)\big(\sqrt{(n-1)^2+4}-\sqrt{(n-2)^2+4}\big)\quad(3)
\end{split}
\end{equation*}
and
\begin{equation*}
\begin{split}
  &\sqrt{(n-2)^2+4}+\big(\sqrt{(n-2)^2+(n-2)^2}-\sqrt{(n-3)^2+(n-3)^2}\big)+\big(\sqrt{(n-2)^2+9}\\
  &-\sqrt{(n-3)^2+9}\big)+(n-5)\big(\sqrt{(n-2)^2+4}-\sqrt{(n-3)^2+4}\big)
\end{split}
\end{equation*}
respectively. By comparing the corresponding terms, we find that the first case is larger than the second, except for item $\sqrt{(n-2)^2+b^2}-\sqrt{(n-3)^2+(b-1)^2}$. However,
\begin{equation*}
\begin{split}
  &\sqrt{(n-2)^2+(n-1)^2}-\sqrt{(n-3)^2+(n-2)^2}-\big(\sqrt{(n-2)^2+(n-2)^2}\\
  &-\sqrt{(n-3)^2+(n-3)^2}\big)\\
  =&\sqrt{(n-2)^2+(n-1)^2}-\sqrt{(n-2)^2+(n-2)^2}-\big(\sqrt{(n-3)^2+(n-2)^2}\\
  &-\sqrt{(n-3)^2+(n-3)^2}\big)\\
  <&\sqrt{(n-2)^2+(n-1)^2}-\sqrt{(n-2)^2+(n-2)^2}\\
  <&\sqrt{(n-1)^2+4}-\sqrt{(n-2)^2+4}
\end{split}
\end{equation*}
where $\sqrt{(n-1)^2+4}-\sqrt{(n-2)^2+4}$ can be contributed in the last part of $(3)$ $(n-4-(n-5)=1)$. Therefore equality holds if and only if $G_n-w \cong L_{n-1}$, $d(u)=n-2, d(v)=n-1, d(p)=3$ and the remaining vertices have degree 2. Which implies $G_n\cong L_n$ and we complete the proof.
\end{proof}

Next we consider the minimum Sombor coindex of two-trees.
\begin{theorem}
Let $G_n$ be a two-tree of order $n\geq2$. Then
\begin{equation*}
\overline{SO}(G_n)\geq\sqrt{2}(n-2)(n-3),
\end{equation*}
with equality holds if and only if $G_n\cong X_n$.
\end{theorem}

\begin{proof}
We prove this result by induction on $n$.

For $n=2,3,4$, $G_n$ can only be isomorphic to $K_2, K_3, X_4$ where $K_i$ is the complete graph of order $i$, respectively, so the equality holds.

Assume that the result holds for $n-1$. Let $w$ be a vertex of $G_n$ with degree 2, then $G_{n}-w$ is a two-tree of order $n-1$. By the induction hypothesis, $\overline{SO}(G_n-w)\geq \overline{SO}(X_{n-1})$ with equality holds if and only if $G_n-w \cong X_{n-1}$. In the following we prove that $\overline{SO}(G_n)\geq \overline{SO}(X_n)$.

Let $N_{G_n}(w)=\{u,v\}$ and $d(u)=a,d(v)=b$.
Then $3\leq a,b\leq n-1$, and
  \begin{equation*}
  \begin{split}
   \overline{SO}(G_n)=&\overline{SO}(G_n-w)+\sum_{u_i\nsim u}\left(\sqrt{a^2+d(u_i)^2}-\sqrt{(a-1)^2+d(u_i)^2}\right)\\
     &+\sum_{v_i\nsim v}\left(\sqrt{b^2+d(v_i)^2}-\sqrt{(b-1)^2+d(v_i)^2}\right)+\sum_{t\nsim w}\sqrt{4+d(t)^2}\\
     \geq&\overline{SO}(X_{n-1})+\sqrt{8}(n-3)\\
     =&\overline{SO}(X_n).
   \end{split}
  \end{equation*}

  If $a=b=n-1$, the first two summations are equal to 0 since no vertices are non-adjacent to $u, v$, and the inequality holds since $d(t)\geq 2$. Hence the equality holds if and only if $a=b=n-1$ and $d(t)=2$, which implies $G_n\cong X_n$.
\end{proof}

Finally, we consider the second minimum Sombor coindex of two-trees.
\begin{lemma}\label{lem2}
  Let $f(x,y)=\sqrt{y^2+x^2}-\sqrt{(y-1)^2+x^2}+\sqrt{4+x^2}$ where $x>0, y\geq3$. Then $f(x,y)$ is monotonically increasing with $x$.
\end{lemma}

\begin{proof}
  The derivative of function $f$ with respect to $x$ is
  \begin{equation*}
      f'(x)=\frac{x}{\sqrt{y^2+x^2}}-\frac{x}{\sqrt{(y-1)^2+x^2}}+\frac{x}{\sqrt{4+x^2}}.
  \end{equation*}

  And $f'(x)>0$ since
  \begin{equation*}
      \frac{1}{\sqrt{4+x^2}}-\frac{1}{\sqrt{(y-1)^2+x^2}}=\frac{\sqrt{(y-1)^2+x^2}-\sqrt{4+x^2}}
      {\sqrt{(y-1)^2+x^2}\sqrt{4+x^2}}\geq0.
  \end{equation*}
\end{proof}
\begin{theorem}
  Let $G_n$ ba a two-tree of order $n\geq5$ and $G_n\ncong X_n$. Then
  \begin{equation*}
    \overline{SO}(G_n)\geq \sqrt{2}(n-3)(n-4)+\sqrt{(n-2)^2+4}+\sqrt{13}(n-4),
  \end{equation*}
  with equality holds if and only if $G_n\cong L_n$.
\end{theorem}

\begin{proof}
We prove this result by induction on $n$.

For $n=5$, $G_n$ can only be isomorphic to $L_5$, so the equality holds.

Assume that the result holds for $n-1$. Let $w$ be a vertex of $G_n$ with degree 2 such that $G_{n}-w \ncong X_{n-1}$, then $G_{n}-w$ is a two-tree of order $n-1$. By the induction hypothesis, $\overline{SO}(G_n-w)\geq \overline{SO}(L_{n-1})$ with equality holds if and only if $G_n-w \cong L_{n-1}$. In the following we prove that $\overline{SO}(G_n)\geq \overline{SO}(L_n)$.

Let $N_{G_n}(w)=\{u,v\}$ and $d(u)=a,d(v)=b$, then $3\leq a,b\leq n-1$.
From the definition of two-tree we know that there must exist a vertex $p$ adjacent to $u$ and $v$ with $d(p)=c\geq3$. Without loss of generality we assume $a\leq b$.

  \textbf{Case 1}: $\max\{b, c\}=c$.

  Then $c\leq n-2$ since $p\nsim w$, hence $a\leq b\leq n-2$, and
  \begin{align*}
      \overline{SO}(G_n)=&\overline{SO}(G_n-w)+\sum_{u_i\nsim u}\left(\sqrt{a^2+d(u_i)^2}-\sqrt{(a-1)^2+d(u_i)^2}\right)\\
      &+\sum_{v_i\nsim v}\left(\sqrt{b^2+d(v_i)^2}-\sqrt{(b-1)^2+d(v_i)^2}\right)+\sum_{t\nsim w}\sqrt{4+d(t)^2}\\
      \geq&\overline{SO}(L_{n-1})+\sum_{u_i\nsim u}\left(\sqrt{a^2+4}-\sqrt{(a-1)^2+4}\right)\\
      &+\sum_{v_i\nsim v}\left(\sqrt{b^2+4}-\sqrt{(b-1)^2+4}\right)
      +\sum_{t\nsim w,t\neq p}\sqrt{4+4}+\sqrt{3^2+4}\\
      \geq&\overline{SO}(L_{n-1})+2\left(\sqrt{(n-2)^2+4}-\sqrt{(n-3)^2+4}\right)\\
      &+2\sqrt{2}(n-4)+\sqrt{13}\\
      =&\sqrt{2}(n-4)(n-5)-\sqrt{(n-3)^2+4}+\sqrt{13}(n-4)\\
      &+2\sqrt{(n-2)^2+4}+2\sqrt{2}(n-4)\quad(4)\\
      >&\overline{SO}(L_n).
  \end{align*}

  From the definition of the Sombor coindex and the construction of $G_n$, we know that the vertices $u_i, v_i$ must equals some $t$ in the last summation $\sum \limits_{t\nsim w}\sqrt{4+d(t)^2}$. Let $x=d(u_i), y=a$ and $x=d(v_i), y=b$ in Lemma \ref{lem2} respectively and the first inequality holds.

  It is not difficult to find that $\sum \limits _{u_i\nsim u}\left(\sqrt{a^2+4}-\sqrt{(a-1)^2+4}\right)$ has $n-1-a$ summation terms, then for $a\leq n-3$, we have
  \begin{equation*}
  \sum_{u_i\nsim u}\left(\sqrt{a^2+4}-\sqrt{(a-1)^2+4}\right) =(n-1-a)(\sqrt{a^2+4}-\sqrt{(a-1)^2+4})>1,
  \end{equation*}
  where the inequality holds since $n-1-a\geq2$ and $\sqrt{a^2+4}-\sqrt{(a-1)^2+4}\geq \sqrt{13}-\sqrt{8}\approx 0.777$.

  If $a=n-2$,
  \begin{equation*}
  \sum_{u_i\nsim u}\left(\sqrt{a^2+4}-\sqrt{(a-1)^2+4}\right) =  \sqrt{(n-2)^2+4}-\sqrt{(n-3)^2+4}<1,
  \end{equation*}
  since $\sqrt{(n-1)^2+4}-\sqrt{(n-2)^2+4}$ represents the difference between the distance from point $(n-1,2)$ to point $(0,0)$ and $(1,0)$. So the second inequality holds.

  The last inequality holds since
  \begin{equation*}
      (4)-\overline{SO}(L_n)=\sqrt{(n-2)^2+4}-\sqrt{(n-3)^2+4}>0.
  \end{equation*}
  \textbf{Case 2}: $\max\{b, c\}=b$.

  Then $b\leq n-1$ and $a\leq n-2$ since $G\ncong X_n$. Similarly we have
\begin{align*}
      \overline{SO}(G_n)=&\overline{SO}(G_n-w)+\sum_{u_i\nsim u}\left(\sqrt{a^2+d(u_i)^2}-\sqrt{(a-1)^2+d(u_i)^2}\right)\\
      &+\sum_{v_i\nsim v}\left(\sqrt{b^2+d(v_i)^2}-\sqrt{(b-1)^2+d(v_i)^2}\right)+\sum_{t\nsim w}\sqrt{4+d(t)^2}\\
      \geq&\overline{SO}(L_{n-1})+\sqrt{a^2+4}-\sqrt{(a-1)^2+4}
      +\sum_{t\nsim w,t\neq p}\sqrt{4+d(t)^2}+\sqrt{13}\\
      \geq&\overline{SO}(L_{n-1})+\sqrt{(n-2)^2+4}-\sqrt{(n-3)^2+4}
      +2\sqrt{2}(n-4)+\sqrt{13}\\
      =&\sqrt{2}(n-3)(n-4)+\sqrt{13}(n-4)+\sqrt{(n-2)^2+4}\\
      =&\overline{SO}(L_n).
  \end{align*}

  Where the equality holds if and only if $a=n-2, b=n-1, c=3$ and the other vertices have degree 2. Hence $\overline{SO}(G_n)\geq \overline{SO}(L_n)$ with equality holds if and only if $G_n\cong L_n$.
\end{proof}

\section{Conclusion}
\hspace{1.5em}In this paper, we investigated the Sombor (co)index of two-trees, which has a very important structure in complex networks. The maximum and second maximum Sombor index, the minimum and second minimum Sombor coindex in two-trees are determined. However, the maximum Sombor coindex, the minimum Sombor index in two-trees are unknown. We conjecture that these two extreme values will be contributed by the following two graphs $H^1_n$ (Figure \ref{H1}), $H^2_n$ (Figure \ref{H2}).
\begin{figure}[htbp]
\centering
\begin{minipage}[t]{0.48\textwidth}
\centering
\includegraphics[width=7cm]{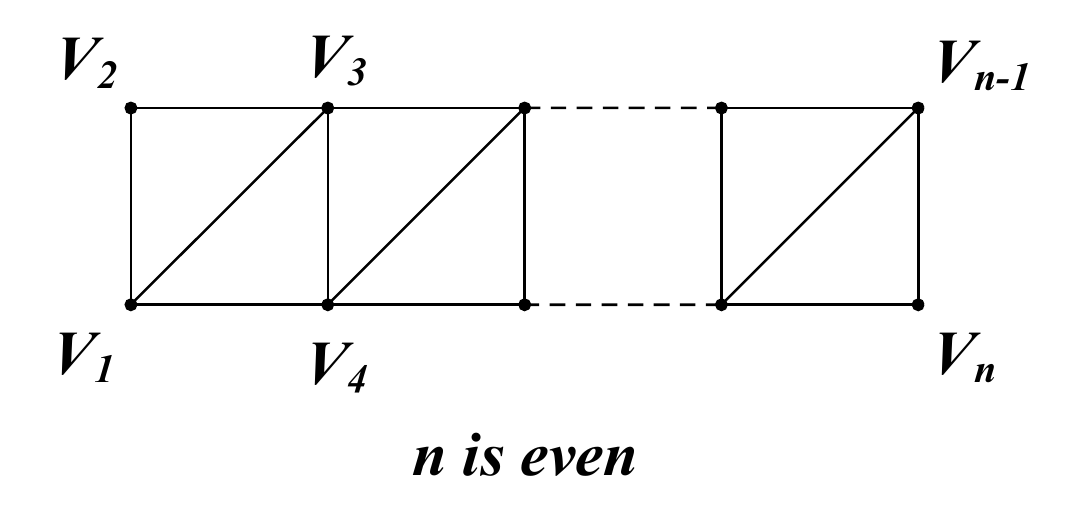}
\caption{The graph $H^1_n$}\label{H1}
\end{minipage}
\begin{minipage}[t]{0.48\textwidth}
\centering
\includegraphics[width=8cm]{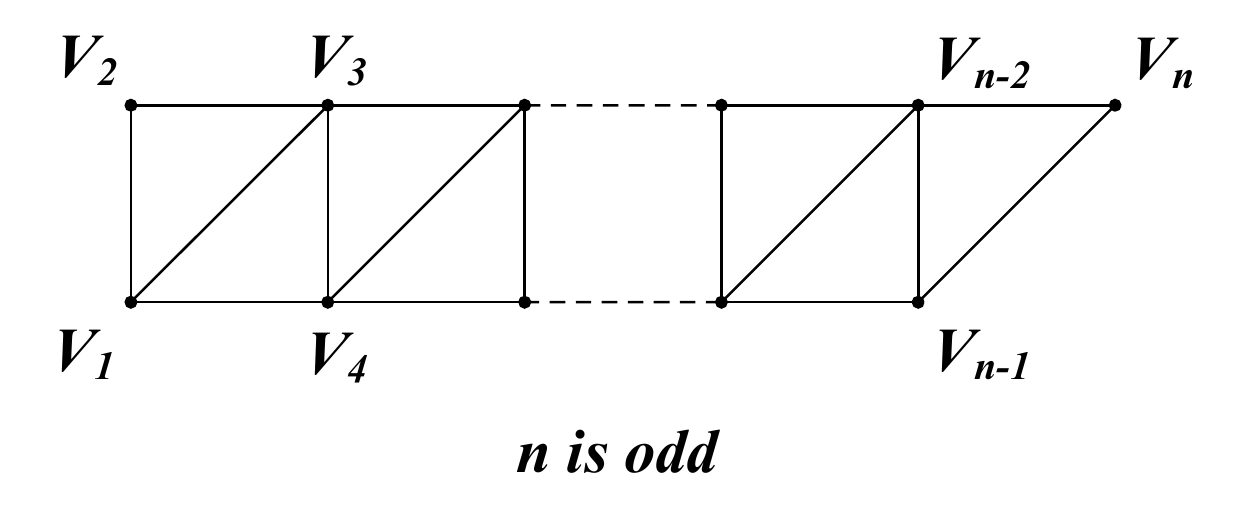}
\caption{The graph $H^2_n$}\label{H2}
\end{minipage}
\end{figure}
\\

\noindent
{\bf Conjecture}\,
Let $G_n$ be a two-tree of order $n$. Then
\begin{equation*}
\begin{split}
&SO(G_n)\geq 6\sqrt{2}n+2\sqrt{13}+4\sqrt{5}+20-33\sqrt{2},\\
&\overline{SO}(G_n)\leq 2\sqrt{2}n^2+(10-26\sqrt{2}n+4\sqrt{5})n+89\sqrt{2}+2\sqrt{13}-20\sqrt{5}-60,
\end{split}
\end{equation*}
with equality holds if and only if $G_n \cong H^i_n$ ($i=1$ if $n$ is even, $i=2$ if $n$ is odd).
\\
\noindent
{\bf Acknowledgement}\,

This work is supported by the National Natural Science Foundation of China (Grant Nos.11971180, 11501139),
the Guangdong Provincial Natural Science Foundation (Grant No. 2019A1515012052).

\end{document}